\documentclass[reqno]{amsart}
\usepackage{latexsym, amsmath, amscd, amssymb, amsthm, bm, mathrsfs}
\usepackage{amsrefs}
\usepackage[colorlinks=true,linkcolor=blue,urlcolor=blue,citecolor=blue]{hyperref}
  
 \newtheorem{lemma}{Lemma}[section]

\newtheorem{cor}[lemma]{Corollary}
\newtheorem{prop}[lemma]{Proposition}

\theoremstyle{definition}

\newtheorem{remark}[lemma]{Remark}
\newtheorem{example}[lemma]{Example}
\theoremstyle{remark}

\DeclareMathOperator{\Ker}{Ker}

\DeclareMathOperator{\Image}{Im}
\DeclareMathOperator{\rk}{rk}

\begin{document}

\title{Envelopes and evolutes}
\author[R.~Piene]{Ragni Piene}
\address{Ragni Piene\\Department of Mathematics\\
University of Oslo\\P.O.Box 1053 Blindern\\NO-0316 Oslo\\Norway}
\email{\href{mailto:ragnip@math.uio.no}{ragnip@math.uio.no}}
\urladdr{\href{http://www.mn.uio.no/math/english/people/aca/ragnip/index.html}
{www.mn.uio.no/math/english/people/aca/ragnip/index.html}}

\date{\today}

\keywords{Envelope, evolute, Thom--Boardman singularities, Thom polynomials}
\subjclass[2020]{Primary 14N15; Secondary 14C17, 58K20}

\begin{abstract}
The study of evolutes of plane curves goes back at least to Huygens, and was continued and extended to space curves by Monge, Darboux, and others. Salmon studied projective curves and surfaces and their evolutes and gave many enumerative formulas for their degrees and number of singularities.

We define envelopes of families of linear spaces in projective space. In order to define evolutes we impose a notion of perpendicularity, which allows us to consider the normal spaces to a projective variety. The evolute of a projective hypersurface is  the envelope of the family of normal lines. For a variety of dimension $r$ in $n$-space, the evolute is defined as the $(n-r)$th ``iterated'' cuspidal locus of the map from the total space of the normal spaces to projective space. Thus the envelope can be interpreted as a $(n-r)$th order Thom--Boardman singularity. Further higher order Thom--Boardman singularities correspond, for a curve in the plane or  in 3-space, to classical objects like the vertices of the curve; for a surface in 3-space, they give the cuspidal curve -- and its cusps -- on the evolute.
Using known formulas for Thom polynomials we are able to verify and generalize many of Salmon's formulas and find new ones.
\end{abstract}

\maketitle

\hfill{\emph{To the memory of L\^e D\~ung Tr\'ang}}

\section{Introduction}

The evolute of a curve in the real Euclidean plane is the locus of its centers of curvature (Huygens 1673, see \cite{Huy}), i.e., the centers of its osculating circles. Alternatively, it can be defined as the envelope of the family of normal lines to the curve. 
A hundred years later, Monge \cite{Monge} considered space curves (``courbes \`a double courbure''). He replaced the centers of the osculating  circles by the  polar lines (``axes polaires'') -- the lines through the center of an osculating circle perpendicular to the osculating plane. The union of these lines forms the ``surface des p\^oles'' (polar surface) of the curve (see \cite{Del}*{3.2, pp.~240--243}). Monge observed that this surface is developable and equal to the envelope of the normal planes to the curve. He showed how to construct infinitely many curves -- ``d\'evelopp\'ees'' (evolutes) -- on this surface, with the property that the tangents of these curves all intersect the given curve; indeed, the tangents are normals to the given curve. In particular, the ``surface des p\^oles''  is \emph{foliated} by the ``d\'evelopp\'ees''. 
These  ``d\'evelopp\'ees'' were also studied by Darboux  \cite{Darb}*{p.~18}.
Monge introduced the envelope of the polar lines, which he called the ``ar\^ete de rebroussement'' (cuspidal edge). Thus the tangents to the cuspidal edge of the polar surface are the polar lines of the given curve.
For more on the history of evolutes of space curves in the 18th and 19th century, see \cite{Del}.

The natural generalization of the evolute of a plane curve to a space curve, in Euclidean space, is to take for the evolute
 the locus of its centers of \emph{spherical} curvature, i.e., the centers of its osculating spheres \cite{P1}*{p.~101}. 
Alternatively, one can define the evolute of a space curves as the cuspidal edge of the envelope of the family of normal planes to the curve. This evolute was called the ``Evolute zweiter Art'' (evolute of the second kind) by Blaschke and Leichtweiss \cite{BL}*{p.~56}. They observed that the evolute is characterized by the fact that its osculating planes are the normal planes to the given curve. Hence the space evolute is a solution to the following interpolation problem: Given a space curve, find a space curve such that its osculating planes are the normal planes to the given curve.
 
Porteous \cite{P1} treated real curves in $\mathbb R^3$ from the point of view of  differential geometry. There are also more recent works such as \cites{Fuchs}, \cite{FT}, \cite{FIRST}.
The \emph{generalized evolute} of a curve in higher dimensional Euclidean space was defined by Romero-Fuster and Sanabria-Codesal \cite{Romero} and Uribe-Vargas \cite{UV} as the locus of the centers of its osculating hyperspheres. 
In  \cite{P1}, Porteous also considered the evolute (which he called the \emph{focal surface}) of a surface in $\mathbb R^3$. In particular he gave a detailed  local description of its singularities. 

Salmon treated the theory of polar surfaces (he called them polar developables) and evolutes of \emph{complex projective} space curves. He observed that the cuspidal edge of the polar developable is the locus of the centers of spherical curvature  \cite{Salmon}*{Art.~372, p.~339}. In particular, he gave formulas for their degrees in terms of the numerical characters of the given  curve. He also studied the evolute (which he called the ``surface of centres'') of a surface in projective 3-space; in particular, he found formulas for the degree and class of the evolute and for the number of umbilics of the surface.

In this paper we consider projective complex varieties of any dimension and define the analogue of generalized evolutes of real curves. The evolutes are defined as ``iterated cuspidal loci'' of the envelope of the family of normal spaces to the given variety. These loci are Thom--Boardman singularities of the map from the total space of the family of normal spaces to the projective space, and hence their classes can be computed from known Thom polynomials.
In Section \ref{envelopes} we define envelopes of families of linear spaces and determine the cycle class of their singular (cuspidal) loci. In Section \ref{evolutes} we introduce a notion of perpendicularity which allows us to define normal spaces to a projective variety. We can therefore consider the family of normal spaces to the given variety and define evolutes. In Section \ref{polarcurves}, we prove numerical formulas for the envelope of the normal spaces to curves in arbitrary dimensional projective space. We show that our formulas agree with those of Salmon in the case of curves in 3-space. Section \ref{polarsurfaces} computes various cycle classes associated to the evolute of a surface, and, in the case of a surface in 3-space, the degree of its evolute, the degree of the cuspidal curve of the evolute, and the number of cusps of the cuspidal curve. Section \ref{osccurves} treats osculating developables of curves in the context of envelopes. 
\medskip

{\bf Acknowledgments.}
Many thanks to Boris Shapiro for fruitful discussions that started at the meeting ``Mørketidens Mattemøte'' at the University of Troms{\o} in January 2023, where our participation was supported by the project Pure Mathematics in Norway, funded by Trond Mohn Foundation and Troms{\o} Research Foundation.

The author would also like to thank the Isaac Newton Institute for Mathematical Sciences, Cambridge, for support and hospitality during the programme ``New equivariant methods in algebraic and differential geometry'' in March 2024, where a part of the work on this paper was undertaken. This work was supported by EPSRC grant EP/R014604/1. Particular thanks go to L\'aszl\'o Feh\'er and Andrzej Weber for helpful discussions. Finally, the author thanks the referee for very useful comments.

\section{Envelopes of families of linear spaces}\label{envelopes}
We fix a $\mathbb C$-vector space $V$ of dimension $n+1$ and consider the corresponding projective $n$-dimensional space $\mathbb P(V)$. (We use the Grothendieck convention, so that points in $\mathbb P(V)$ are surjections $V\to \mathbb C$.)
Let $X$ be a projective nonsingular variety of dimension $r$, and $f\colon X\to \mathbb P(V)$ a morphism. Assume $\mathcal F$ is a rank $n-r+1$ locally free sheaf on $X$ and $V_X\to \mathcal F$ a surjective map. 
Then $\mathcal F$ gives rise to a family of linear $(n-r)$-spaces 
\[\psi\colon \mathbb P(\mathcal F) \to \mathbb P(V),\]
 where $\psi$ is the composition of the inclusion $\mathbb P(\mathcal F)\subset X\times \mathbb P(V)$ with the projection on the second factor $X\times \mathbb P(V) \to \mathbb P(V)$. Note that $\dim \mathbb P(\mathcal F)=\dim \mathbb P(V)$.
 
The \emph{envelope} $E_\mathcal F$ of the family of linear spaces is the branch locus of the morphism $\psi$.
From a singularities-of-mappings point of view, this means that the envelope is the image by $\psi$ of the Thom--Boardman singularity
 \[\Sigma^1:=\Sigma^1(\psi)=\{y\in \mathbb P(\mathcal F) \,|\, \rk d\psi_y \le n-1\}.\]
Here we shall also consider the higher order Thom--Boardman singularities $\Sigma^{1,1}=\Sigma^1(\psi|_{\Sigma^1})$ and  $\Sigma^{1,1,1}=\Sigma^1(\psi|_{\Sigma^{1,1}})$.

The classes of these loci are given by Thom polynomials in the Chern classes $\overline c_i:=c_i(\psi^*T_{\mathbb P(V)}-T_{\mathbb P(\mathcal F)})$, provided the map $\psi$ is suitably generic, in particular provided the loci $\Sigma^1$, $\Sigma^{1,1}$, and $\Sigma^{1,1,1}$ have the expected codimension. 
In what follows, we shall always assume this holds.

Recall that the Thom polynomial for $\Sigma^1$ is $\overline c_1$, that for $\Sigma^{1,1}$ is $\overline c_1^2+\overline c_2$ \cite{Ronga}, and that for 
$\Sigma^{1,1,1}$ is $\overline c_1^3+3\overline c_1\overline c_2+2\overline c_3$ \cite{Rim}. 
(The Thom polynomial for $\Sigma^{1,1,1,1}$ was
computed by Gaffney, Porteous, and Ronga \cite{Gaf}*{Thm.~2.2, p.~407}:
\[\overline c_1^4 + +6\overline c_1^2\overline c_2+ 9\overline c_1\overline c_3 +2\overline c_2^2 + 6\overline c_4,\]
and those for $\Sigma^{1,1,\dots,1}$ up to codimension 8 by Rimanyi \cite{Rim}*{Thm.~5.1, p.~508}.)
\medskip

Set $c_i':=c_i(T_{\mathbb P(V)})$ and $c_i:=c_i(T_{\mathbb P(\mathcal F)})$. 

\begin{lemma}\label{chern}
We have
\begin{eqnarray*}
\overline c_1   &=&  \psi^*c_1'-c_1\\
\overline c_2 &=&  \psi^*c_2'-c_2-\psi^*c_1'c_1+c_1^2\\
\overline c_3  &=&  \psi^*c_3'-c_3-\psi^*c_2'c_1-\psi^*c_1'c_2+\psi^*c_1'c_1^2+2c_1c_2-c_1^3.
\end{eqnarray*}
\end{lemma}

\begin{proof}
This is straightforward, from the fact that $\overline c_i$ is the $i$th degree term of $(\sum \psi^*c_i')(\sum c_i)^{-1}$.
\end{proof}

Let $\pi\colon \mathbb P(\mathcal F) \to X$ denote the structure map.

\begin{lemma}\label{cF}
We have
\[ c(\Omega_{\mathbb P(\mathcal F)}^1)=\pi^*c(\Omega^1_X)c(\pi^*\mathcal F \otimes \mathcal O_{\mathbb P(\mathcal F)}(-1)).
\]
\end{lemma}

\begin{proof}
The exact sequence
\[ 0\to \pi^* \Omega_X^1 \to \Omega_{\mathbb P(\mathcal F)}^1 \to \Omega^1_{\mathbb P(\mathcal F)/X} \to 0
\]
gives $c(\Omega_{\mathbb P(\mathcal F)}^1)=\pi^*c(\Omega^1_X)c(\Omega_{\mathbb P(\mathcal F)/X}^1)$,
and the sequence
\[ 0 \to \Omega_{\mathbb P(\mathcal F)/X}^1 \otimes \mathcal O_{\mathbb P(\mathcal F)}(1) \to \mathcal P^1_{\mathbb P(\mathcal F)/X}(1)=\pi^*\mathcal F \to  \mathcal O_{\mathbb P(\mathcal F)}(1) \to 0
\]
gives $c(\Omega_{\mathbb P(\mathcal F)/X}^1)=c(\pi^*\mathcal F \otimes \mathcal O_{\mathbb P(\mathcal F)}(-1))$.
\end{proof}

\begin{lemma}\label{chern2}
We have
\begin{eqnarray*}
\psi^*c_i'&=&\binom{n+1}ic_1(\mathcal O_{\mathbb P(\mathcal F)}(1))^i\\
c_1&=&-\pi^*c_1(\Omega^1_X)-\pi^*c_1(\mathcal F)+(n-r+1)c_1(\mathcal O_{\mathbb P(\mathcal F)}(1))\\
c_2&=& \pi^*(c_2(\Omega^1_X)+c_2(\mathcal F)+c_1(\Omega^1_X)c_1(\mathcal F))
-(n-r+1)\pi^*c_1(\Omega^1_X)c_1(\mathcal O_{\mathbb P(\mathcal F)}(1))\\
&&-(n-r)\pi^*c_1(\mathcal F)c_1(\mathcal O_{\mathbb P(\mathcal F)}(1))
+\binom{n-r+1}2c_1(\mathcal O_{\mathbb P(\mathcal F)}(1))^2\\
c_3&=&-\sum_{i=0}^3(-1)^{3-i}\binom{n-r+1-i}{3-i}\pi^*c_i(\mathcal F)c_1(\mathcal O_{\mathbb P(\mathcal F)}(1))^{3-i} \\
&&-\pi^*c_1(\Omega_X^1)\sum_{i=0}^2(-1)^{2-i}\binom{n-r+1-i}{2-i}\pi^*c_i(\mathcal F)c_1(\mathcal O_{\mathbb P(\mathcal F)}(1))^{2-i} \\
&&-\pi^*c_2(\Omega_X^1)\sum_{i=0}^1(-1)^{1-i}\binom{n-r+1-i}{1-i}\pi^*c_i(\mathcal F)c_1(\mathcal O_{\mathbb P(\mathcal F)}(1))^{1-i} \\
&& -\pi^*c_3(\Omega_X^1).
\end{eqnarray*}
\end{lemma}

\begin{proof}
The exact sequence 
\[ 0\to \Omega_{\mathbb P(V)}^1(1) \to V_{\mathbb P(V)} \to \mathcal O_{\mathbb P(V)}^1(1) \to 0
\]
gives 
\[c(T_{\mathbb P(V)})=c(V^\vee\otimes \mathcal O_{\mathbb P(V)}(1)) =\bigl(1+c_1(\mathcal O_{\mathbb P(V)}(1))\bigr)^{n+1}.\]
Since $\psi^*\mathcal O_{\mathbb P(V)}(1)=\mathcal O_{\mathbb P(\mathcal F)}(1)$,
 the first equality follows. 

For the other three, we use Lemma \ref{cF} and the fact that
\[c_k(\pi^*\mathcal F\otimes \mathcal O_{\mathbb P(\mathcal F)}(-1))=\sum_{i=0}^k (-1)^{k-i} \binom{n-r+1-i}{k-i}c_i(\pi^*\mathcal F)c_1(\mathcal O_{\mathbb P(\mathcal F)}(1))^{k-i}.\] 
It follows that $c_i=c_i(T_{\mathbb P(\mathcal F)})=(-1)^ic_i(\Omega_{\mathbb P(\mathcal F)}^1)$ is as stated  in the lemma.
\end{proof}

Let  $s_{i}(\mathcal F)=\pi_*c_1(\mathcal O_{\mathbb P(\mathcal F)}(1))^{n-r+i}$ denote the $i$th Segre class of $\mathcal F$.

\begin{prop}\label{env}
The class of the envelope $E_\mathcal F$  is 
\[[E_\mathcal F]=\psi_*\bigl(\pi^* c_1(\Omega^1_X)+\pi^*c_1(\mathcal F)+rc_1(\mathcal O_{\mathbb P(\mathcal F)}(1))\bigr)\cap [\mathbb P(V)],\]
and its degree is
\[\deg E_\mathcal F=\bigl(c_1(\Omega_X^1)s_{r-1}(\mathcal F)+c_1(\mathcal F)s_{r-1}(\mathcal F)+rs_r(\mathcal F)\bigr)\cap [X].
\]
\end{prop}

\begin{proof}
The class of $E_\mathcal F$ is given by the Thom polynomial $\overline c_1=\psi^*c_1'-c_1$, hence
\[ [E_\mathcal F]=\psi_*\bigl((n+1-n+r-1)\psi^*c_1(\mathcal O_{\mathbb P(V)}(1)+\pi^*c_1(\Omega^1_X)+\pi^*c_1(\mathcal F)\bigr)\cap [\mathbb P(V)],\]
which shows the first equality, since $\psi^*\mathcal O_{\mathbb P(V)}(1)=\mathcal O_{\mathbb P(\mathcal F)}(1)$.

The degree of $E_\mathcal F$ is
\[\deg E_\mathcal F=c_1(\mathcal O_{\mathbb P(V)}(1))^{n-1} \cap [E_\mathcal F],\]
which is equal to
\[\psi_*\bigl(\pi^*(c_1(\Omega^1_X)+c_1(\mathcal F))c_1(\mathcal O_{\mathbb P(\mathcal F)}(1))^{n-1}+rc_1(\mathcal O_{\mathbb P(\mathcal F)}(1))^n\bigr)\cap [\mathbb P(V)].\]
We can replace $\psi_*$ by $\pi_*$. Then, using the projection formula,
 the second equality follows.
\end{proof}

\begin{cor}\label{envcurve}
Assume $X$ is a curve, so that $r=1$. Then
\[\deg E_\mathcal F =\bigl(c_1(\Omega_X^1)+2c_1(\mathcal F)\bigr)\cap [X].\]
\end{cor}

\begin{proof}
In this case, $s_{r-1}(\mathcal F)=s_0(\mathcal F)$ and $s_r(\mathcal F)=s_1(\mathcal F)=c_1(\mathcal F)$.
\end{proof}

\begin{cor}\label{envsurf}
Assume $X$ is a surface, so that $r=2$. Then
\[\deg E_\mathcal F =\bigl(c_1(\Omega_X^1)c_1(\mathcal F)+3c_1(\mathcal F)^2-2c_2(\mathcal F)\bigr)\cap [X] .\]
\end{cor} 

\begin{proof}
We have $s_1(\mathcal F)=c_1(\mathcal F)$ and $s_2(\mathcal F)=c_1(\mathcal F)^2-c_2(\mathcal F)$.
\end{proof}
\medskip

The ramification locus of the restriction of the map of sheaves to $\Sigma^1$ is the second order Thom--Boardman singularity locus $\Sigma^{1,1}$. Denote by $C_\mathcal F$ the image by $\psi$ of this ramification locus, i.e., $C_\mathcal F:=\psi_*(\Sigma^{1,1})$ is the branch locus of $\psi|_{E_{\mathcal F}}$. In case that $X$ is a curve in 3-space, $C_\mathcal F$ is the \emph{cuspidal edge} of the developable  surface $E_{\mathcal F}$.

\begin{prop}\label{cusp}
The ``cuspidal edge'' $C_\mathcal F\subset \mathbb P(V)$ of the envelope $E_\mathcal F$ has class
\begin{multline*} [C_\mathcal F]=\psi_*
\bigl(\pi^*(2c_1(\Omega^1_X)^2+3c_1(\Omega^1_X)c_1(\mathcal F)+2c_1(\mathcal F)^2
-c_2(\Omega^1_X)-c_2(\mathcal F))\\
+(3r-1)\pi^*c_1(\mathcal F)c_1(\mathcal O_{\mathbb P(\mathcal F)}(1))
+3r\pi^*c_1(\Omega^1_X)c_1(\mathcal O_{\mathbb P(\mathcal F)}(1))\\+\frac{3r^2-r}2 c_1(\mathcal O_{\mathbb P(\mathcal F)}(1))^2\bigr)\cap [\mathbb P(V)],
\end{multline*}
and its degree is
\begin{multline*}
\deg C_\mathcal F=\bigl((2c_1(\Omega^1_X)^2+3c_1(\Omega^1_X)c_1(\mathcal F)+2c_1(\mathcal F)^2
-c_2(\Omega^1_X)-c_2(\mathcal F))s_{r-2}(\mathcal F)\\
+((3r-1)c_1(\mathcal F)
+3rc_1(\Omega^1_X))s_{r-1}(\mathcal F)+\frac{3r^2-r}2 s_r(\mathcal F)\bigr)\cap [X].
\end{multline*}

\end{prop}

\begin{proof}
The Thom polynomial for $\Sigma^{1,1}$ is $\overline c_1^2+\overline c_2$. From Lemma \ref{chern} we get
\[\overline c_1^2+\overline c_2=\psi^*{c_1'}^2 + 2c_1^2-3\psi^*c_1'c_1+\psi^*c_2'-c_2.\]
Using Lemma \ref{chern2}, this gives the expressions for $[C_\mathcal F]$ and $\deg C_\mathcal F$.
\end{proof}

\begin{prop}\label{cuspscusp}
Set $h:=c_1(\mathcal O_{\mathbb P(\mathcal F)}(1))=\psi^*c_1(\mathcal O_{\mathbb P(V)}(1)$, $w_i:=\pi^*c_i(\Omega^1_X)$, and $f_i:=\pi^*c_i(\mathcal F)$. 
The class of the cuspidal locus $\kappa_\mathcal F:=\psi_*(\Sigma^{1,1,1})$ of the cuspidal edge $C_\mathcal F$ is
\begin{multline*}\textstyle
[\kappa_\mathcal F]=\psi_*\bigl(
(17\binom{r}3 +12\binom{r}2 +r)h^3
+\bigl((17\binom{r}2 + 6r)w_1
 \textstyle +  (17\binom{r}2 +r +2)f_1\bigr)h^2\\
 +\bigl(11rw_1^2+(17r-5)w_1f_1+(11r-7)f_1^2-5rw_2-(5r-4)f_2\bigr)h\\
 +6w_1^3+11w_1^2f_1+11w_1f_1^2+6f_1^3-7w_1w_2
 -7w_2f_1-5w_1f_2-7f_1f_2+2w_3+2f_3\bigr)\cap [\mathbb P(V)].
\end{multline*}
\end{prop}

\begin{proof} The class of $\kappa_\mathcal F$ is given by the Thom polynomial of $\Sigma^{1,1,1}$, which is equal to $\overline c_1^3+3\overline c_1\overline c_2+2\overline c_3$.  Lemma \ref{chern} and Lemma \ref{chern2} give 
\begin{multline*}
\textstyle \overline c_1=rc_1(\mathcal O_{\mathbb P(\mathcal F)}(1))+\pi^*(c_1(\Omega^1_X)+c_1(\mathcal F) )\\
\textstyle \overline c_2=\binom{r}2 c_1(\mathcal O_{\mathbb P(\mathcal F)}(1))^2+\pi^*(rc_1(\Omega^1_X)+(r-1)c_1(\mathcal F))c_1(\mathcal O_{\mathbb P(\mathcal F)}(1))\\
+\pi^*(c_1(\Omega^1_X)^2+c_1(\Omega^1_X)c_1(\mathcal F)+c_1(\mathcal F)^2
-c_2(\Omega^1_X)-c_2(\mathcal F))\\
\textstyle \overline c_3=\binom{r}3c_1(\mathcal O_{\mathbb P(\mathcal F)}(1))^3 +\pi^*(\binom{r}2 c_1(\Omega_X^1)+\binom{r-1}2 c_1(\mathcal F))c_1(\mathcal O_{\mathbb P(\mathcal F)}(1))^2 \\+\pi^*\bigl(rc_1(\Omega_X^1)^2+(r-1)c_1(\Omega_X^1)c_1(\mathcal F)
+(r-2)c_1(\mathcal F)^2-rc_2(\Omega_X^1)\\
-(r-2)c_2(\mathcal F)\bigr)c_1(\mathcal O_{\mathbb P(\mathcal F)}(1))
+\pi^*\bigl(c_3(\mathcal F)+c_3(\Omega_X^1)+c_1(\Omega_X^1)c_2(\mathcal F)\\
-2c_1(\Omega_X^1)c_2(\Omega_X^1)-2c_1(\Omega_X^1)c_2(\mathcal F)\-2c_2(\Omega_X^1)c_1(\mathcal F)
-2c_1(\mathcal F)c_2(\mathcal F)\\
+c_1(\Omega_X^1)^3+c_1(\Omega_X^1)^2c_1(\mathcal F)+c_1(\Omega_X^1)c_1(\mathcal F)^2+c_1(\mathcal F)^3\bigr),
\end{multline*}
from which the proposition follows. Note that the expression is polynomial in $r$ and independent of $n$.
\end{proof}

\begin{cor}\label{degcuspcurve}
Assume $X$ is a curve, so that $r=1$. Then
\[\deg C_\mathcal F=
3\bigl(c_1(\Omega^1_X)+c_1(\mathcal F)\bigr)\cap [X].
\]
If $n\ge 3$, then
\[\deg \kappa_\mathcal F=2\bigl(3c_1(\Omega^1_X)+2c_1(\mathcal F)\bigr)\cap [X] .\]
\end{cor}

\begin{proof}
From Proposition \ref{cusp}, with $r=1$, we get the formula for $\deg C_\mathcal E$, and
\[[\kappa_\mathcal F]= \psi_*\bigl(c_1(\mathcal O_{\mathbb P(\mathcal F)}(1))^3 +6\pi^*c_1(\Omega^1_X)c_1(\mathcal O_{\mathbb P(\mathcal F)}(1))^2
+3\pi^*c_1(\mathcal F) c_1(\mathcal O_{\mathbb P(\mathcal F)}(1))^2\bigr)\cap [\mathbb P(V)]
\]
from which the formula for $\deg \kappa_\mathcal E$ follows by observing that  $s_1(\mathcal F)=c_1(\mathcal F)$ and $s_2(\mathcal F)=c_1(\mathcal F)^2-c_2(\mathcal F)$.
\end{proof}

\begin{cor}\label{degcuspcurvesurface}
Assume $X$ is a surface, so that $r=2$. Then
\[
\deg C_\mathcal F= \bigl(2c_1(\Omega^1_X)^2 -c_2(\Omega^1_X)+ 9 c_1(\Omega^1_X)c_1(\mathcal F) +12c_1(\mathcal F)^2
-6c_2(\mathcal F)\bigr)\cap [X]
\]
and
\[
\deg \kappa_\mathcal F=2\bigl(11 c_1(\Omega^1_X)^2 -5  c_2(\Omega^1_X)+29 c_1(\Omega^1_X)c_1(\mathcal F) +25c_1(\mathcal F)^2-10c_2(\mathcal F)\bigr)\cap [X].
\]
\end{cor}

\begin{proof}
The first formula follows from Proposition \ref{cusp} with $r=2$. For the second, Proposition \ref{cuspscusp} gives
\begin{multline*}
[\kappa_\mathcal F]=\psi_*\bigl(
14c_1(\mathcal O_{\mathbb P(\mathcal F)}(1))^3 +\pi^*(29 c_1(\Omega^1_X)
+ 21 c_1(\mathcal F))
c_1(\mathcal O_{\mathbb P(\mathcal F)}(1))^2+ \\\pi^*( 22  c_1(\Omega^1_X)^2 
+29 c_1(\Omega^1_X)c_1(\mathcal F)
+15 c_1(\mathcal F)^2
-10 c_2(\Omega^1_X) - 6c_2(\mathcal F))c_1(\mathcal O_{\mathbb P(\mathcal F)}(1))\bigr)\cap [\mathbb P(V)].
\end{multline*}
To get the formula for the degree, use again $s_1(\mathcal F)=c_1(\mathcal F)$ and $s_2(\mathcal F)=c_1(\mathcal F)^2-c_2(\mathcal F)$.
\end{proof}

\section{Evolutes}\label{evolutes}

In order to define the normal spaces to a projective variety, we need to equip $\mathbb P(V)$  with a ``Euclidean structure'', by which we mean a notion of perpendicularity between linear spaces.  Let 
$V'$ be a vector space of dimension $n$ and $V\to V'$ a surjection, so that $\mathbb P(V')\subset \mathbb P(V)$ is a hyperplane. Set $H_\infty:=\mathbb P(V')$, the hyperplane at infinity. 
Let $Q_\infty\subset H_\infty$ be a non-singular quadric hypersurface. If $L\subset H_\infty$ is an $i$-dimensional linear space, we let $L^\perp \subset H_\infty$ denote the $(n-2-i)$-dimensional linear space which is \emph{polar} to $L$ with respect to $Q_\infty$ \cite{reci}*{Sec.~4, p.~244}.
If $L,L'\subset \mathbb P(V)$ are linear spaces, we let $\langle L,L'\rangle$ denote their linear span.

Assume that $X\subset \mathbb P(V)$ is a nonsingular variety of dimension $r$ and that $X$ is in general position with respect to $H_\infty$ and $Q_\infty$, i.e., $H_\infty$ and $X$,  and $Q_\infty$  and $X\cap H_\infty$, intersect transversally.  Let $\mathcal P^1_X(1)$ denote the sheaf of principal parts of $\mathcal O_X(1)$ of order 1, which is a vector bundle of rank $r+1$. The first jet map $V_X\to \mathcal P^1_X(1)$ is surjective, and the fibers  $T_P:=\mathbb P(\mathcal P^1_X(1))_P\subset  \mathbb P(V)$ are the projective tangent spaces to $X$.
Set $\mathcal K^1:=\Ker (V_X\to \mathcal P^1_X(1))$, so that $\mathcal K^1$ is the conormal bundle of $X$ in $\mathbb P(V)$ twisted by $\mathcal O_X(1)$. Then $\mathcal E:=(\mathcal K^1)^\vee \oplus \mathcal O_X(1)$ is the Euclidean normal bundle to $X$, and there is a natural surjective map $V_X \to \mathcal E$ (using the isomorphism $(V')^\vee \cong V'$ induced by the quadric $Q_\infty$)
\cite{reci}*{Prop.~4.1, p.~245}.  

The fibers $N_P:= \mathbb P(\mathcal E)_P\subset \mathbb P(V)$ are the normal spaces to $X$, i.e., the
$(n-r)$-spaces perpendicular to the tangent spaces $T_P$. Indeed, $N_P=\langle P, (T_P\cap H_\infty)^\perp \rangle$ (see \cite{reci}*{Sec.~4}).

Let $E_\mathcal E$ denote the envelope of the family $ \mathbb P(\mathcal E) \subset X\times \mathbb P(V) \to \mathbb P(V)$ of  normal $(n-r)$-spaces to $X$, i.e., $E_\mathcal E$ is the branch locus of the composed map $\psi\colon \mathbb P(\mathcal E) \to \mathbb P(V)$. We define the \emph{evolute} of $X$ to be the image by $\psi$ of the $(n-r)$th order singularity locus $\Sigma^{1,\dots,1}$, where the number of 1's is $n-r$. In particular, if $X$ is a hypersurface, then its evolute is equal to the envelope of its normal lines.

If $X\subset \mathbb P(V)$ is a nonsingular hypersurface of degree $d$, in general position with respect to $H_\infty$ and $Q_\infty$, we recover Trifogli's formula (see \cite{Trifogli}*{Thm.~2}, \cite{Polar_rev}*{Ex.~10, p.~149}):
\[\textstyle \deg E_\mathcal E=d(d-1)\bigl((n-1)(d-1)^{n-2}+2\sum_{i=0}^{n-2}(d-1)^i\bigr).\]
This follows from Proposition \ref{env}, with $\mathcal F=\mathcal E=(\mathcal K^1)^\vee \oplus \mathcal O_X(1)$. Since $X$ is a hypersurface of degree $d$, we have $c_1(\Omega_X^1)=c_1(\mathcal O_X(d-n-1))$ and $\mathcal K^1=\mathcal O_X(-d)\otimes \mathcal O_X(1)=\mathcal O_X(-d+1)$. 

Note that, for $d=2$, we get $\deg E_\mathcal E=6(n-1)$,
which gives the very classical formulas for the degree of the evolute of a plane conic ($n=2$) and of a quadric surface ($n=3$).

\begin{remark}\label{Nash}
If $f:X\to \mathbb P(V)$ is a morphism where $X$ is not necessarily nonsingular and $f$ is not necessarily an embedding, we may replace $X$ by its Nash transform $\nu:\widetilde X\to X$ and $V_X\to \mathcal P^1_X(1)$ by its Nash quotient $V_{\widetilde X}\to \mathcal P^1$ \cite{reci}*{Sec.~3}.
In this case, the genericity assumption is that $H_\infty$ intersects each Whitney stratum of $f(X)$ transversally, and that $Q_\infty$ intersects each Whitney stratum of $f(X)\cap H_\infty$ transversally \cite{reci}*{Sec.~4, p.~245}.
\end{remark}

\section{Evolutes of curves}\label{polarcurves}
Let $f: X\to \mathbb P(V)$ be a morphism from a nonsingular curve $X$, and assume $f$ is birational onto its image. Assume $f(X)$ is in general position with respect to $H_\infty$ and $Q_\infty$ (see Remark \ref{Nash}).
Let $d=f^*c_1(\mathcal O_{\mathbb P(V)}(1))\cap [X]$ denote the degree of $X$, and $g$ its genus.

 Let $k_0$ denote the degree of the ramification locus of $f$, i.e., $k_0$ is the ``weighted number of cusps'' of $f(X)$. Since $X$ is a nonsingular curve, the image  $\mathcal P^1$ of the first jet map $V_X\to \mathcal P^1_X(1)$, is a vector bundle, of rank 2. The cokernel of the jet map is equal to $\Omega^1_{X/\mathbb P(V)}(1)$. Hence we have an exact sequence 
 \[0\to \mathcal P^1\to \mathcal P^1_X(1)\to \Omega^1_{X/\mathbb P(V)}(1) \to 0.\]
 The ramification locus is defined by the 0th Fitting ideal of $\Omega^1_{X/\mathbb P(V)}$, equivalently of 
 $\Omega^1_{X/\mathbb P(V)}(1)$ \cite{numchar}*{p.~481}. Hence we get  
 \[c_1(\mathcal P^1)\cap [X] =c_1(\mathcal P^1_X(1))\cap [X]-k_0=2d+2g-2-k_0,\]
 where the last equality follows from the fact that $c_1(\mathcal P^1_X(1))=
 c_1(\Omega^1_X)+2c_1(\mathcal O_X(1))$.
 Set $\mathcal K^1:=\Ker(V_X \to \mathcal P^1)$. Then $\mathcal E=\mathcal (\mathcal K^1)^\vee \oplus \mathcal O_X(1)$  is the Euclidean normal bundle (see Remark \ref{Nash}).

\begin{prop}\label{poldev}
The class of the envelope $E_\mathcal E$ is 
\[[E_\mathcal E]=\psi_*\bigl(\pi^*(( c_1(\mathcal P^1)+c_1(\Omega_X^1) +c_1(\mathcal O_X(1))) + c_1(\mathcal O_{\mathbb P(\mathcal E)}(1))\bigr )\cap [\mathbb P(V)],\]
and its degree is
\[\deg E_\mathcal E=6(d+g-1)-2k_0.
\]
\end{prop}

\begin{proof}
The proposition follows from Corollary \ref{envcurve} with $\mathcal F=\mathcal E$, using $c_1(\mathcal E)=-c_1(\mathcal K^1)+c_1(\mathcal O_X(1))=c_1(\mathcal P^1)+c_1(\mathcal O_X(1))$ and $c_1(\mathcal P^1)\cap [X]=2d+2g-2-k_0$.

\end{proof}

\begin{remark} When $n=2$, so that $f(X)$ is a plane curve, the envelope is the \emph{evolute} of the curve. See \cite{ev} for a thorough treatment of this case, both in the real and complex situation.
\end{remark}

\begin{prop}\label{spaceev}
The class of the ``cuspidal edge''  $C_\mathcal E$ of the envelope is
\[
\psi_*\bigl (2\pi^*c_1(\mathcal P^1)c_1(\mathcal O_{\mathbb P(\mathcal E)}(1))\\
+3\pi^*c_1(\Omega^1_X)c_1(\mathcal O_{\mathbb P(\mathcal E)}(1))+3c_1(\mathcal O_{\mathbb P(\mathcal E)}(1))^2\bigr)\cap [\mathbb P(V)],
\]
and its degree is
\[\deg C_\mathcal E=3(3d+4g-4-k_0).\]
If $n\ge 3$, the class of the cuspidal locus $\kappa_\mathcal E$ of the cuspidal edge is
\[
\psi_*\bigl (3\pi^*c_1(\mathcal P^1)c_1(\mathcal O_{\mathbb P(\mathcal E)}(1))^2\\
+6\pi^*c_1(\Omega^1_X)c_1(\mathcal O_{\mathbb P(\mathcal E)}(1))^2+c_1(\mathcal O_{\mathbb P(\mathcal E)}(1))^3\bigr)\cap [\mathbb P(V)],
\]
and its degree is 
\[\deg \kappa_\mathcal E=4\bigl(3d+5g-5-k_0\bigr).
\]
\end{prop}

\begin{proof}
This follows from Corollary \ref{degcuspcurve} with $\mathcal F=\mathcal E$.
\end{proof}

Consider now the case $n=3$. Then the cuspidal edge $C_\mathcal E$  of the envelope $E_\mathcal E$  is the \emph{evolute}  of $X$. 
We shall see that the formulas for the degrees of the envelope and the evolute are in accordance with Salmon's formulas \cite{Salmon}*{Footnote, p.~341}. He states that the envelope $E_\mathcal E$  has degree $3m+n$, where $m$ is the degree of the curve $X$ and $n$ is the number of osculating planes to the curve that pass through a given point. 
Let $k_1$ denote the (weighted) number of inflection points of the curve. We have $m=d$ and $n=3(d+2g-2)-2k_0-k_1$ \cite{numchar}*{Thm.~3.2, p.~481}.
Hence
$3m+n=3d+3(d+2g-2)-2k_0-k_1=6(d+g-1)-2k_0-k_1$. Salmon assumes that the curve has no inflection points, i.e., that $k_1=0$, so in this case his formula agrees with the one of Proposition \ref{poldev}. In a footnote he says that if the curve has $\theta$ inflection points, then $n$ should be replaced by $n+\theta$, so his formula would read $3m+n+\theta=3d+3(d+2g-2)-2k_0-k_1+k_1=6(d+g-1)-2k_0$, in agreement with Proposition \ref{poldev}

To find the degree of the envelope $E_\mathcal E$, Salmon \cite{Salmon}*{Footnote, p.~341} considers the intersection of $E_\mathcal E$ with the plane $H_\infty$ at infinity. He observes that at each of the $m$ intersection points $X\cap H_\infty$, the polar line at that point is contained in $H_\infty$, with multiplicity $3$. He says that the remaining part of $E_\mathcal E\cap H_\infty$ is a curve of degree $n$. We can understand this as follows: Let $L\subset H_\infty$ be a (general) line, and $L^\perp\in H_\infty$ its polar point (with respect to the conic $Q_\infty$). There are $n$ osculating planes to $C$ that pass through the point $L^\perp$. Let $O_P$, $P\in X$, be such a plane, and let $T_P\subset O_P$ denote the tangent line to $X$ at $P$. Then $T_P\cap H_\infty \subset O_P\cap H_\infty$, hence $(O_P\cap H_\infty)^\perp\subset (T_P\cap H_\infty)^\perp$ and $(O_P\cap H_\infty)^\perp \subset L$. It remains to see that the line joining $P$ and $(O_P\cap H_\infty)^\perp$ is the polar line to $C$ at $P$. Now the normal plane $N_P$ to $C$ at $P$ is the plane spanned by $P$ and $(T_P\cap H_\infty)^\perp$. The polar line is the intersection of $N_P$ with the normal plane $N_{P'}$ at an infinitesimally near point $P'$ to $P$. It therefore suffices to see that the intersection of $(T_P\cap H_\infty)^\perp$ with $(T_{P'}\cap H_\infty)^\perp$ for $P'$ infinitely near $P$ is equal to $(O_P\cap H_\infty)^\perp$. But this follows from the fact that the span of $T_P$ and $T_{P'}$ is equal to $O_P$.
Finally, one can check, by using a local parameterization of $X$, that the polar line at an inflection point lies in the plane $H_\infty$. So the intersection of $E_\mathcal E$ with the plane at infinity consists of $m$ lines with multiplicity $3$, $\theta$ lines, and a curve of degree $n$

Salmon's formula for the degree of the evolute $C_\mathcal E$ is $5m+\alpha$, where $\alpha$ denotes the number $k_2=4(d+3(g-1))-3k_0-2k_1$ of hyperosculating planes to the curve \cite{numchar}*{Thm.~3.2, p.~481}, i.e., $5m+\alpha=5d+4(d+3(g-1))-3k_0-2k_1=3(3d+4g-4-k_0)-2k_1$. This agrees with our formula in Proposition \ref{spaceev} when $k_1=0$. But the number of hyperosculating points should be weighted, so that $\alpha =2k_1+k_2$. This gives $5m+\alpha=5d+4(d+3(g-1))-3k_0-2k_1+2k_1=3(3d+4g-4-k_0)$, again in agreement with the formula in Proposition \ref{spaceev}.

To find the degree of the evolute $C_\mathcal E$, Salmon \cite{Salmon}*{Footnote, p.~341} asserts that the degree of the curve in the dual projective space $\mathbb P(V^\vee)$ consisting  of the osculating planes to the evolute is $m+r$, where $m$ is the degree of $X$ and $r$ is the rank of $X$, i.e., the degree of its tangent developable. The reason for this is that the osculating planes to the evolute are the normal planes to $X$ \cite{BL}.

Salmon does not give a formula for $\deg \kappa_\mathcal E$. However, since $E_\mathcal E$ is the tangent developable of the curve $C_\mathcal E$, we have the formula \cite{numchar}*{Thm.~(3.2), p.~481}
\[\deg E_\mathcal E=2\deg C_\mathcal E +2g-2-\deg \kappa_\mathcal E,\]
since $C_\mathcal E$ and $X$ have the same genus $g$.
This gives
\[\deg \kappa_\mathcal E=2\cdot 3(3d+4g-4-k_0)+2g-2 - 6(d+g-1)+2k_0=4(3d+5g-5-k_0),\]
which checks with Proposition \ref{spaceev}.

\begin{example}
Let $X\subset \mathbb P(V)$ be a twisted cubic. 
The envelope has degree $\deg E_\mathcal E=12$. Its evolute $C_\mathcal E$ has degree 15 and has 16 cusps.
\end{example}

\begin{remark}
The cusps of the evolute correspond to the \emph{vertices} of the given curve. For a curve in $n$-space, their number can be computed as the degree of  $\Sigma^{1,\dots,1}$, where the number of 1's is $n$.
\end{remark}

\section{Evolutes of surfaces}\label{polarsurfaces}
Assume $\mathbb P(V)$ has a Euclidean structure as described in the beginning of Section \ref{evolutes}, and that $X\subset \mathbb P(V)$ is a nonsingular surface which is in general position with respect to $H_\infty$ and $Q_\infty$. Then  the Euclidean normal bundle is $\mathcal E=(\mathcal K^1)^\vee\oplus \mathcal O_X(1)$, where  $\mathcal K^1=\Ker (V_X\to \mathcal P^1_X(1))$ is equal to the 
conormal bundle of $X$ in $\mathbb P(V)$ twisted by $\mathcal O_X(1)$:
The composed map $\psi: \mathbb P(\mathcal E)\subset X\times \mathbb P(V) \to \mathbb P(V)$ is the family of normal $(n-2)$-spaces to $X$.

Let again $E_\mathcal E$ denote the envelope of the normal spaces of $X$, $C_\mathcal E$ its cuspidal locus, and $\kappa_\mathcal E$ the cuspidal locus of $C_\mathcal E$.

\begin{prop}\label{polarofsurf}
We have
\begin{multline*}
\deg E_\mathcal E =\bigl(2c_1(\Omega_X^1)^2 +2 c_2(\Omega^1_X) +18c_1(\Omega_X^1)c_1(\mathcal O_X(1))
+ 30 c_1(\mathcal O_X(1))^2 \bigr )\cap [X]\\
\deg C_\mathcal E = \bigl (17 c_1(\Omega_X^1)^2 +5 c_2(\Omega^1_X) +102 c_1(\Omega_X^1)c_1(\mathcal O_X(1)) +138 c_1(\mathcal O_X(1))^2 \bigr )\cap [X]\\\
\deg \kappa_\mathcal E = 2\bigl (55 c_1(\Omega_X^1)^2 +5 c_2(\Omega^1_X) + 266 c_1(\Omega_X^1)c_1(\mathcal O_X(1)) + 310  c_1(\mathcal O_X(1))^2 \bigr )\cap [X]
\end{multline*}
\end{prop}

\begin{proof}
We apply Corollaries \ref{envsurf} and \ref{degcuspcurvesurface} with $\mathcal F=\mathcal E=(\mathcal K^1)^\vee \oplus \mathcal O_X(1)$. We have
\begin{multline*}
c_1(\mathcal E)=-c_1(\mathcal K^1)+c_1(\mathcal O_X(1))=c_1(\mathcal P^1_X(1))+c_1(\mathcal O_X(1))=c_1(\Omega_X^1)+4c_1(\mathcal O_X(1)).
\end{multline*}
Now $c(\mathcal K^1)=c(\mathcal P^1_X(1))^{-1}$, which gives
\[c_2(\mathcal K^1)=c_1(\Omega^1_X)^2-c_2(\Omega_X^1)+4c_1(\Omega_X^1)c_1(\mathcal O_X(1))+6c_1(\Omega_X^1)^2.\]
Hence we get
\begin{multline*}
c_2(\mathcal E)=c_2(\mathcal K^1)-c_1(\mathcal K^1)c_1(\mathcal O_X(1))\\
= c_1(\Omega_X^1)^2-c_2(\Omega_X^1)+5c_1(\Omega_X^1)c_1(\mathcal O_X(1))+9c_1(\mathcal O_X(1))^2.
\end{multline*}
\end{proof}

\begin{cor}
Assume $n=3$ and let $d$ denote the degree of $X$. Then we have
\begin{eqnarray*}
\deg E_\mathcal E &=& 2d(d-1)(2d-1)\\
\deg C_\mathcal E&= &2d(d-1)(11d-16)\\
\deg \kappa_\mathcal E&=& 4d(30d^2-97d+78).
\end{eqnarray*}
\end{cor}

\begin{proof}
We apply Proposition \ref{polarofsurf} with
 $c_1(\Omega^1_X)=(d-4)c_1(\mathcal O_X(1))$, $c_2(\Omega^1_X)=(d^2-4d+6)c_1(\mathcal O_X(1))^2$, and $c_1(\mathcal O_X(1))^2 \cap [X]=d$. 
\end{proof}

Salmon calls the evolute $E_\mathcal E$ of a surface in 3-dimensional space for  the ``surface of centres,'' or ``centro-surface'' \cite{Salmon2}*{Art.~507, p.~148}, since it is the locus of spherical curvature centers, the \emph{focal} points. 
Since there are two principal ``curvature directions'' at each point of the surface, there are two focal points on each normal line to the surface. Moreover, each normal line is tangent to the evolute at each of the two focal points, hence are bitangents to the evolute.
 Considered as a surface in $\mathbb P(\mathcal E)$, the evolute is a $2\colon 1$ cover of $X$. For a detailed local study of the evolute in the case of a real surface, see \cite{P1}*{}.

For $d=2$, our formula gives $\deg E_\mathcal E=12$, in accordance with Salmon \cite{Salmon}*{Art.~206, p.~179}, who actually computes its equation. Moreover, in the fifth edition of Salmon's book \cite{Salmon2}*{Art.~511, p.~151}, he gives the formula for $\deg E_\mathcal E$ above and also the formula $2d(d^2-d-1)$  for the degree of the dual surface -- the class of $E_\mathcal E$  \cite{Salmon2}*{Art.~509, p.~150}. 
For $d=2$, this gives 4, as is also stated in \cite{Salmon}*{Art.~199, p.~171}. In the general case, to find the class of the evolute, Salmon starts by determining the number of the bitangents of the evolute, which are the same as the normals to the surface $X$, that pass through a given point \cite{Salmon2}*{Art.~507, p.~149}. This number is the socalled Euclidean distance degree of $X$ \cite{Draisma}*{}, and is equal to $d+d(d-1)+d(d-1)^2=d(d^2-d+1)$. He shows this by taking the point to be a point in the plane at infinity. 
From this, and a consideration of the points on a surface, the normals at which meet a given line, he deduces the formula for the class of $E_\mathcal E$. For the degree of $E_\mathcal E$, he gives two alternative approaches. One, attributed to Darboux,  uses the theory of congruences of lines, the lines being the bitangents of the evolute, which are also the normals of the surface $X$. The other consists in studying the intersection of $E_\mathcal E$ with the plane at infinity.

\begin{remark}
There will be a finite number of points on the surface where the two focal points on the normal line coincide.
These points are the \emph{umbilic points} (see  \cite{Salmon}*{Art.~106, p.~85} and \cite{P1}*{p.~159}).
Salmon showed \cite{Salmon}*{footnote on p.~263} (cited in \cite{BRW}) that the number of umbilics,
 disregarding points at $\infty$,  is equal to
\[2d(5d^2-14d+11).\footnote{This formula is a correction to the formula $10d^3-25d^2+16d$ given on p.~229 in the 1865 edition of Salmon's book \cite{Salmon1865}. The difference between the two formulas is the number $3d(d-2)$ of inflection points of the plane curve $X\cap H_\infty$, erroneously included in the 1865 formula. (In a footnote, Salmon refers to Voss, Math. Annalen, IX, 1876.)}\]
Porteous considered affine surfaces in $\mathbb R^3$ and observed that the points on the evolute above the umbilics (where the two focal points coincide) are points of singularity type $D_4$  for the map from the total space of the normals to the surface to $\mathbb R^3$. One could hope that this could carry over to the present situation of a complex projective surface, and that the map $\psi \colon \mathbb P(\mathcal E)\to \mathbb P(V)$ would be Lagrangian. If so, the number of umbilics could be computed by evaluating the 
the corresponding Thom polynomial $\widetilde Q_{21}=\tilde c_1\tilde c_2-2\tilde c_3$, where the $\tilde c_i$ are the Chern classes of the virtual bundle $\Omega_{\mathbb P(\mathcal E)}^1 - \psi^*\Omega_{\mathbb P(V)}^1$ \cite{MPW}*{ p.~76}. 
However, this computation does not give  Salmon's formula, and the discrepancy is rather large.

Another approach, as suggested in \cite{P2}*{p.~124}, again for an \emph{affine} surface $X_0\subset \mathbb R^3$, would be to consider the  set of coinciding focal points as the  $\Sigma^{2,2}$ locus of the map
\[X_0\times \mathbb R^3\to \mathbb R\times \mathbb R^3,\] 
sending $(x,u)$ to $(|x-u|^2,u)$, where $|x-u]^2$ is the square of the (Euclidean) distance function.
One could then try to evaluate the corresponding Thom polynomial  \cite{Ronga} for a projectivized version of this map.
\end{remark}

\section{Osculating developables of curves}\label{osccurves}
Assume $X$ is a nonsingular curve and $f\colon X\to \mathbb P(V)$ a map which is birational onto its image. Let $d=f^*c_1(\mathcal O_{\mathbb P(V)}(1))\cap [X]$ denote the degree of $X$, and $g$ its genus. 
Denote by $\mathcal P^m_X(1)$ the sheaf of principal parts of order $m$ of $\mathcal O_X(1):=f^*\mathcal O_{\mathbb P(V)}(1)$, and let $a^m\colon V_X\to \mathcal P^m_X(1)$ denote the $m$th jet map. Let $\mathcal P^m:=\Image a^m$ denote the $m$th osculating bundle \cite{numchar}*{Prop.~(2.1), p.~478}, and set $\mathcal K^m:=\Ker a^m$, so that we have exact sequences
\[0\to \mathcal K^m \to V_X \to \mathcal P^m\to 0.\]
 The \emph{$m$th  osculating developable} $D^m_X$ of $X$ is the image of the composed map
\[\mathbb P(\mathcal P^m)\subset X\times \mathbb P(V)\to \mathbb P(V).\]
The degree of $D_X^m$ is
\[\deg D^m_X=c_1(\mathcal O_{\mathbb P(\mathcal P^m)}(1))^{m+1}\cap [\mathbb P(\mathcal P^m)]
=\pi_*c_1(\mathcal O_{\mathbb P(\mathcal P^m)}(1))^{m+1} \cap [X],\]
which is equal to 
\[s_1(\mathcal P^m)\cap [X]=c_1(\mathcal P^m)\cap [X].\]
From the exact sequences
\[0\to S^i \Omega^1_X \otimes \mathcal O_X(1)=(\Omega^1_X)^{\otimes i} \otimes \mathcal O_X(1) \to \mathcal P^i_X(1) \to \mathcal P^{i-1}_X(1)\to 0,\]
it follows that
\[ \textstyle c_1(\mathcal P^m_X(1))= \binom{m+1}2 c_1(\Omega^1_X)+(m+1)c_1(\mathcal O_X(1)).\]
The stationary indices $k_i$ of the curve are defined pointwise via local parameterizations of the curve \cite{numchar}*{p.~481}. Intuitively, $k_0$ is the (weighted) number of cusps, $k_1$ the (weighted) number of inflectional tangents, \dots, and $k_{n-1}$ the (weighted) number of hyperosculating hyperplanes. These numbers satisfy the following global formulas \cite{numchar}*{Thm.~(3.2), p.~481}:
\[ \textstyle (c_1(\mathcal P^m_X(1))-c_1(\mathcal P^m))\cap [X]=\sum_{i=0}^{m-1} (m-i)k_i.\]
Hence
\[ \textstyle \deg D^m_X=c_1(\mathcal P^m)\cap [X]= \binom{m+1}2 (2g-2) +(m+1)d  -\sum_{i=0}^{m-1}(m-i)k_i.\]

It is well know that in the case $n=2$, when $X$ is a plane curve, the envelope of the tangents to $X$ is equal to $X$ union its inflectional tangents. When $n=3$, the envelope of the osculating planes to $X$ is equal to the tangent developable of $X$ union its hyperosculating planes. More generally, we have the following result, which generalizes and extends to the complex case a similar result of Ishikawa for real curves \cite{Ishikawa}*{Thm.~1, p.~604}.

\begin{prop}
The envelope $E_{\mathcal P^{n-1}}$ of the family of osculating hyperplanes 
\[\psi\colon \mathbb P(\mathcal P^{n-1})\to \mathbb P(V)\]
to $X$ is equal to  its $(n-2)$th osculating developable $D^{n-2}_X$ union its hyperosculating hyperplanes.
\end{prop}

\begin{proof}
Let $t$ be a local parameter on $X$ and $t,u_1,\dots,u_{n-1}$ local parameters on $\mathbb P(\mathcal P^{n-1})$. Let $r(t)$ be a local parameterization of $f\colon X\to \mathbb P(V)$.
The $(n-1)$th jet map $a^{n-1}\colon V_X\to \mathcal P^1_X(1)$ is given locally by the matrix
\[A^{n-1}(t):=\left( \begin{array}{c}
r(t)\\
r'(t)\\
\vdots\\
r^{(n-1)}(t)
\end{array}\right).
\]
 Above points on $X$ where the map $a^{n-1}$ is surjective, the map $\psi\colon \mathbb P(\mathcal P^{n-1})\to \mathbb P(V)$ is
 given locally by \[(t,u_1,\ldots,u_{n-1})\mapsto r(t)+u_1r'(t)+\cdots +u_{n-1}r^{(n-1)}(t).\]
Its differential $d\psi$ is given by
\[dA^{n-1}(t,u):=\left( \begin{array}{c}
r'(t)+u_1r''(t)+\cdots +u_{n-1}r^{(n)}(t)\\
r'(t)\\
\vdots\\
r^{(n-1)}(t)
\end{array}\right)\cong
\left(\begin{array}{c}
u_{n-1}r^{(n)}(t)\\
r'(t)\\
\vdots\\
r^{(n-1)}(t)
\end{array}\right).
\]
By assumption, the rank of $A^{n-1}(t)$ is $n$, hence the 
rank of $d\psi$ is $<n$ when $u_{n-1}r^{(n)}(t)=0$.   If $u_{n-1}=0$, the point lies in the $(n-2)$th osculating space to $X$ at $r(t)$, hence in $D^{n-2}_X$. If $r^{(n)}(t)=0$, the point $r(t)$ is a point of hyperosculation, and $d\psi$ has rank $< n$ at all points above it in $\mathbb P(\mathcal P^{n-1})$, i.e., at all points in the $(n-1)$th osculating space at the point $r(t)$.

In the case that we are considering a point $r(t)$ where $a^{n-1}$ is not surjective, say for $t=0$, then the inclusion $\mathcal P^{n-1}\subseteq \mathcal P^{n-1}_X(1)$ is given by a matrix which has rank $< n$ for $t=0$. The map $V_X\to \mathcal P^{n-1}$ is given by a modification of the matrix $A^{n-1}(t)$. A local study as in \cite{numchar}*{pp.478--479} shows that we get the same conclusion as in the previous case.
\end{proof}
\medskip

Let us check that the numerical formulas agree.
It follows from Proposition \ref{envcurve} that
\[\deg E_{\mathcal P^{n-1}}=2g-2+2c_1(\mathcal P^{n-1})\cap [X].\]
hence
\begin{multline*}\textstyle \deg E_{\mathcal P^{n-1}}= 2g-2+2n\bigl(d+(n-1)(g-1)-\sum_{i=0}^{n-2} (n-1-i)k_i\bigr)\\
\textstyle =2\bigl(nd+(n^2-n+1)(g-1)-\sum_{i=0}^{n-2} (n-1-i)k_i\bigr).
\end{multline*}
From what we have seen,
\[\textstyle
\deg D_X^{n-2}= (n-1)(d+(n-2)(g-1))-\sum_{i=0}^{n-3}(n-2-i)k_i.\]
We also know that \cite{numchar}*{Thm.~(3.2), p.~481}
\[ \textstyle k_{n-1}= (n+1)(d+n(g-1))-\sum_{i=0}^{n-2}(n-i)k_i.
\]
It follows that $\deg E_{\mathcal P^{n-1}}=\deg D_X^{n-2}+k_{n-1}$.

\end{document}